\documentclass[a4paper,reqno,11pt]{amsart}
\usepackage{amsmath}
\usepackage{amssymb}
\usepackage{cases}
\allowdisplaybreaks[4]

\newtheorem{thm}{Theorem}[section]

\newtheorem{lem}[thm]{Lemma}

\newtheorem{cor}[thm]{Corollary}

\newtheorem{conj}[thm]{Conjecture}
\newtheorem{claim}{Claim}

\theoremstyle{definition}
\newtheorem{definition}[thm]{Definition}

\theoremstyle{remark}
\newtheorem{remark}[thm]{Remark}

\numberwithin{equation}{section}

\newcommand{\bQ}{\mathbb{Q}}

\newcommand\OO{{\mathcal{O}}}
\newcommand{\rounddown}[1]{\lfloor{#1}\rfloor}
\newcommand{\roundup}[1]{\lceil{#1}\rceil}

\newcommand\Vol{\text{\rm Vol}}

\newcommand\mult{{\rm{mult}}}
\newcommand\Nklt{{\rm{Nklt}}}
\newcommand\Nlc{{\rm{Nlc}}}
\newcommand\lct{{\rm{lct}}}
\newcommand\ulct{{\rm{ulct}}}

\begin{document}

\title{On birational boundedness of Fano fibrations}
\date{November 9, 2016, revised version}
\author{Chen Jiang}
\address{Graduate School of Mathematical Sciences, the University of Tokyo,
3-8-1 Komaba, Meguro-ku, Tokyo 153-8914, Japan.}
\email{cjiang@ms.u-tokyo.ac.jp}
\curraddr{Kavli IPMU (WPI), UTIAS, The University of Tokyo, Kashiwa, Chiba 277-8583, Japan.}
\email{chen.jiang@ipmu.jp}

\thanks{The author was supported by Grant-in-Aid for JSPS Fellows (KAKENHI No. 25-6549).}

\begin{abstract}
We investigate birational boundedness of Fano varieties and Fano fibrations.
We establish an inductive step towards birational boundedness of Fano fibrations via conjectures related to boundedness of Fano varieties and Fano fibrations. 
As corollaries, we provide approaches towards birational boundedness and boundedness of anti-canonical volumes of varieties of $\epsilon$-Fano type.
Furthermore, we show birational boundedness of $3$-folds of $\epsilon$-Fano type.
\end{abstract} 

\maketitle
\pagestyle{myheadings} \markboth{\hfill  C. Jiang
\hfill}{\hfill On birational boundedness of Fano fibrations\hfill}

\tableofcontents

\section{Introduction}
Throughout this paper,  we work over the field of complex numbers $\mathbb{C}$.
See Subsection \ref{section notation} for notation and conventions.

A normal projective variety $X$ is \emph{of $\epsilon$-Fano type} if there exists an effective $\bQ$-divisor $B$ such that $(X, B)$ is an $\epsilon$-klt log Fano pair. 

We are mainly interested in the boundedness of varieties of $\epsilon$-Fano type. 
Our motivation  is the following conjecture due to A. Borisov, L. Borisov, and V. Alexeev.

\begin{conj}[{BAB}$_n$]
Fix an integer $n>0$, $0<\epsilon<1$.
Then the set of all
$n$-dimensional varieties of $\epsilon$-Fano type  is bounded.
\end{conj}
 {BAB}$_n$ is one of the most important conjectures in birational geometry and it is related to the termination of flips. Besides, since varieties of Fano type form a fundamental class in birational geometry according to Minimal Model Program, it is very interesting to understand the basic properties of this class, such as boundedness.

{BAB}$_2$ was proved  by Alexeev \cite{AK2} with a simplified argument by Alexeev--Mori \cite{AM}. But BAB$_{\geq 3}$ is still open. There are only  some partial boundedness results (cf.  \cite{BB, KMM92, K, KMMT, AB}).

 As the  approach to this conjecture, 
we are also interested in the following conjecture, where we consider birational boundedness instead of boundedness.

\begin{conj}[{\rm {BBAB}$_n$}]
Fix an integer $n>0$, $0<\epsilon<1$.
The set of all
$n$-dimensional varieties of  $\epsilon$-Fano type is birationally bounded.
\end{conj}

Unfortunately, {BBAB}$_{\geq 3}$ is still open, but
{BBAB}$_2$ holds even without the assumption $\epsilon>0$, since  all  surfaces of Fano type are rational (cf. \cite{AM}). However, in dimension three and higher, it is necessary to assume  $\epsilon>0$ due to counterexamples constructed by Lin \cite{Lin} and Okada \cite{Okada1, Okada2}. 

Besides the boundedness of the families, we are also interested in the boundedness of special invariants of the families. One of the most interesting invariants is the anti-canonical volume.

\begin{conj}[WBAB$_n$]
Fix an integer $n>0$, $0<\epsilon<1$. 
Then there exists a  number $M(n,\epsilon)$ depending only
on $n$ and $\epsilon$ with the following property:

If  $X$ is an $n$-dimensional variety of $\epsilon$-Fano type, then 
${\rm Vol}(X, -K_X) \leq M(n,\epsilon).$
\end{conj}
WBAB$_{\leq 3}$ was proved by the author \cite{Jiang} very recently, while  WBAB$_{\geq 4}$ is still open. 

See Subsection \ref{section conjectures} for more conjectures related to boundedness, such as TBAB, GA, and LCTB.

The main goal of this paper is to investigate BBAB$_n$, especially, to prove BBAB$_3$ (Corollary \ref{cor BBAB3}). As shown in \cite{Jiang}, according to Minimal Model Program, it suffices to investigate varieties of $\epsilon$-Fano type with a Mori fibration  (see Theorem \ref{thm MFS}).
We define the concept of an  $(n, d, a, \epsilon)$-Fano fibration which is a natural generalization of a variety of $\epsilon$-Fano type with a Mori fibration, see Definition \ref{ndae}. 

It is expected that the boundedness of fibrations follows from that of bases and general fibers, and some additional boundedness information on the ambient spaces.
The following is the main theorem of this paper.

\begin{thm}\label{thm main}Fix integers $n>m>0$ and $d>0$,  a rational number $a\geq 0$, $0<\epsilon<1$. Assume {\rm GA}$_{n-m}$, {\rm LCTB}$_{n-1}$,  and ${\rm BrTBAB}_{n-m}$ hold.  Then there exist a positive integer $N(n,d, a, \epsilon)$  and a number $V(n,d, a, \epsilon)$ depending only on $n$, $d$, $a$, and  $\epsilon$, satisfying the following property:

If $(f:X\to Z, B, H)$ is an almost-extremal $(n,d, a,\epsilon)$-Fano fibration with $\dim Z=m$, then 
\begin{itemize}
\item[(i)]  $|-r_{n-m}K_X+N(n, d, a, \epsilon)f^*(H)|$ is ample and gives a birational map;

\item[(ii)] $\Vol(X, -r_{n-m}K_X+N(n, d, a, \epsilon)f^*(H))\leq V(n, d, a, \epsilon)$.
\end{itemize}
In particular, the set of such $X$ forms a birationally bounded family.
\end{thm}
Here $r_{n-m}$ is an integer such that for any  $(n-m)$-dimensional terminal Fano variety $Y$,  
$|-r_{n-m}K_Y|$ gives a birational map.
The existence of such integer is implied by  ${\rm BrTBAB}_{n-m}$, see Conjecture \ref{conj vb}(2). See Subsection \ref{section conjectures} for conjectures assumed in the theorem. 

As corollaries, we establish inductive steps towards {\rm BBAB}$_n$ and  {\rm WBAB}$_n$.

\begin{cor}\label{cor BBAB}
Assume {\rm BAB}$_{\leq n-1}$, {\rm BTBAB}$_n$, {\rm LCTB}$_{n-1}$, {\rm GA}$_{\leq n-1}$, and ${\rm S}_n$ hold. Then {\rm BBAB}$_n$ holds.
\end{cor}

\begin{cor}\label{cor WBAB}
Assume {\rm BAB}$_{\leq n-1}$, {\rm WTBAB}$_n$, {\rm GA}$_{\leq n-1}$, and ${\rm S}_n$ hold. Then {\rm WBAB}$_n$ holds.
\end{cor}

Note that Corollary \ref{cor WBAB} was indicated in \cite{Jiang}, but not clearly stated.

As the most interesting corollaries, we prove {BBAB}$_3$ and {\rm WBAB}$_3$ unconditionally by proving the conjectures we need in lower dimension.
\begin{thm}\label{thm LCTB2}
{\rm LCTB}$_2$ holds.
\end{thm}
\begin{cor}\label{cor BBAB3}
{\rm BBAB}$_3$ holds.  That is, for $0<\epsilon<1$, the set of all
$3$-folds of  $\epsilon$-Fano type is birationally bounded.
\end{cor}
\begin{cor}[\cite{Jiang}]\label{cor WBAB3}
 {\rm WBAB}$_3$ holds.  That is, for  $0<\epsilon<1$, there exists a number $M(3,\epsilon)$ such that for a
$3$-fold $X$ of  $\epsilon$-Fano type, $\Vol(X, -K_X)\leq M(3, \epsilon)$.
\end{cor}

\subsection{Conjectures and historical remarks}\label{section conjectures}
In this subsection, we collect conjectures related to boundedness of Fano varieties.

It is enough interesting to consider the boundedness of terminal Fano varieties.
\begin{conj}
Fix an integer $n$. 
\begin{enumerate}
\item {\rm ({TBAB}$_n$)} The set of all
$n$-dimensional $\bQ$-factorial terminal Fano varieties of Picard number one is bounded.

\item {\rm ({BTBAB}$_n$)} The set of all
$n$-dimensional $\bQ$-factorial terminal Fano varieties of Picard number one is birationally bounded.

\end{enumerate}
\end{conj}
Note that {TBAB}$_{3}$ was  proved by Kawamata \cite{K}.

\begin{conj}\label{conj vb}Fix an integer $n$. 
\begin{enumerate}
\item {\rm ({WTBAB}$_n$)}
There exists a  number $M_0(n)$ depending only
on $n$ such that if  $X$ is an $n$-dimensional $\bQ$-factorial terminal Fano variety of Picard number one, then 
$(-K_X)^n \leq M_0(n).$


\item {\rm ({BrTBAB}$_n$)}
There exists an integer $r_{n}$ depending only
on $n$ such that if  $X$ is an $n$-dimensional  terminal Fano variety, then 
$|-r_{n}K_X|$ gives a birational map.

 \end{enumerate}
\end{conj}
Note that boundedness of the family naturally implies boundedness of anti-canonical volumes  and birationality by generic flatness and Noetherian induction. It is easy to see that $M_0(1)=2$, $M_0(2)=9$, $r_1=1$, and $r_2=3$. Some effective results were obtained, for example, in \cite{ProkK, CJ} which show that we may take $M_0(3)=64$ and $r_3=97$. However, all the boundedness results for terminal Fano $3$-folds heavily rely on classification of terminal singularities in dimension three. So it is interesting to develop methods not depending on classification and work for higher dimension.

Another interesting invariant is $\alpha$-invariant (see Subsection \ref{section notation} for definition). It measures the singularities of log Fano pairs.
In \cite{Jiang}, we formulated the following conjecture on lower bound of $\alpha$-invariants.
\begin{conj}[{GA}$_n$]\label{gac}
Fix an integer $n>0$ and $0<\epsilon<1$. 
Then there exists a   number $\mu(n,\epsilon)>0$ depending only on $n$ and $\epsilon$ with the following property:

 If $(X, B)$  is an  $\epsilon$-klt log Fano pair and $X$ is an $n$-dimensional  terminal Fano variety, then
$\alpha(X,B)\geq \mu(n,\epsilon).$
\end{conj}
Very recently, GA$_{2}$ was proved by the author \cite{Jiang} under general setting when $X$ itself need not to be Fano. However, GA$_{n}$ is not likely to be implied  by BAB$_n$ since the boundary $B$ is involved. It is even not clear if GA$_{n}$ is true for a fixed variety $X$ (and all possible boundaries $B$).

We also propose the following conjecture, which states that the log canonical threshold of the boundary is bounded from below uniformly. 
\begin{conj}[{LCTB}$_n$]\label{conj LCT}
Fix integers $n>0$ and $d>0$, a rational number $a\geq 0$, $0<\epsilon<1$. Then there exists a number  $\lambda(n,d, a, \epsilon)>0$ depending only on $n$, $d$, $a$, and  $\epsilon$, satisfying the following property:

If $(f:X\to Z, B, H)$ is an almost-extremal $(n,d, a,\epsilon)$-Fano fibration, then $(X, (1+t)B)$ is klt for $0<t\leq \lambda(n,d, a, \epsilon)$.
\end{conj}

This conjecture is totally new, and even {LCTB}$_2$ is unknown before. 
It is somehow very technical, but very important in this paper.
We will prove {LCTB}$_2$ in this paper (Theorem \ref{thm LCTB2}).

It is interesting to make a comparison between GA and LCTB. In both conjectures, we are trying to measure the singularities of the pairs by log canonical thresholds of certain divisors. But they go to two different directions. In GA, we consider divisor $G\sim_\bQ -(K_X+B)$, which is ample, but we need to consider all such divisors. In LCTB, we consider only  the boundary, but without any positivity.

We also expect the following conjecture on the images of varieties of $\epsilon$-Fano type via Mori fibrations.
\begin{conj}[${\rm S}_{n}$]
Fix an integer $n>0$ and $0<\epsilon<1$. Then there exists a number $\delta(n, \epsilon)>0$ depending only on $n$ and $\epsilon$ such that if $X$ is an $n$-dimensional variety  of $\epsilon$-Fano type with a Mori fibration $X\to Z$, then $Z$ is of $\delta(n, \epsilon)$-Fano type.
\end{conj}

Note that ${\rm S}_{n}$ is a special consequence of Shokurov's conjecture in Birkar \cite{Birkar}, where he proved ${\rm S}_{n}$ for the case $\dim X-\dim Z=1$ unconditionally.

Finally, as a trivial remark, all conjectures mentioned above hold for $n=1$.
Also, all conjectures are known to be true for $n=2$ except {LCTB}$_2$.

It is worth to mention that, very recently, Birkar \cite{Birkar1, Birkar2} treated several conjectures above and claimed a proof of BAB conjecture using different but much stronger technique.

\subsection{Sketch of the proof}
We explain the idea of the proof of Corollary \ref{cor BBAB}. By the idea of \cite{Jiang}, maybe well known to experts, to prove  birational boundedness of varieties of $\epsilon$-Fano type, it is enough to consider those with a Mori fibration structure, see Theorem \ref{thm MFS}. 

Now we consider an $\epsilon$-klt log Fano pair $(X,B)$ with a Mori fibration $f:X\to Z$. 

To prove the birational boundedness of $X$, as a well known strategy, it suffices to find a Weil divisor $D$ on $X$ such that $|D|$ gives a birational map on $X$, and the volume of $D$ is bounded from above uniformly (cf. \cite[Lemma 2.4.2(2)]{HMX13}). 

A natural candidate of this divisor $D$ is $-mK_X$,  the pluri-anti-canonical divisor. 
For example, to prove the birational boundedness of smooth projective varieties of general type, pluri-canonical divisor $mK_X$ is considered, cf. \cite{HM, Takayama, Tsuji}. But the behavior of anti-canonical divisors is totally different from canonical divisors, for example,  after taking higher models, the bigness of anti-canonical divisors is not preserved. Another candidate is the adjoint pluri-anti-canonical divisor $-m(K_X+B)$. For example, to prove the birational boundedness of log canonical pairs of general type, adjoint pluri-canonical divisor $m(K_X+B)$ is considered, cf. \cite{HMX14}. However, to deal with log canonical pairs of general type, one need always to assume some good condition (DCC) on the coefficients of $B$, on the other hand, for log Fano pairs, we should not assume any condition on the coefficients of $B$, which makes the problem more complicated.
In fact, it is somehow very difficult even to find a uniform number $m$ such that $|-mK_X|$ or $|-m(K_X+B)|$ is nonempty.

Our idea is to make use of the Mori fibration structure $f:X\to Z$, assuming that it is not trivial, i.e., $\dim Z>0$. In this case, we may find a very ample divisor $H$ on $Z$. The degree of $H$ is bounded if $Z$ is bounded, which is guaranteed by assuming S$_n$ and BAB$_{\leq n-1}$. Then we may consider the divisor $-rK_X+mf^*(H)$ instead of $-mK_X$ for some fixed positive integer $r$. 

Note that $-rK_X+mf^*(H)$ is ample for $m$ sufficiently large. If $-rK_X+mf^*(H)$ is ample for some uniform $m$, then by vanishing theorem, it is easy to lift sections from a general member $X_1$ in $|f^*(H)|$ to $X$. This allows us to cut down the dimension of $Z$ by $H$ and reduce to the case when $Z$ is a point, see Lemma \ref{lemma birational}. But one difficulty appears here, that is, $(X_1, B|_{X_1})$ is no long log Fano. Hence for induction, we need to deal with a larger category of varieties. This is why we define the concept of $(n, d, a,\epsilon)$-Fano fibrations which is a generalization of $\epsilon$-klt log Fano pairs with a Mori fibration.  
The induction step works well for this larger category.

Then we need to find a uniform number $L$ such that $-K_X+Lf^*(H)$ is ample. To this end, for a curve $C$ generating an extremal ray of  $\overline{NE}(X)$, we need to give a lower bound for $-K_X\cdot C$. Note that 
$$
-K_X\cdot C> \frac{1}{t}(K_X+(1+t)B)\cdot C,
$$
by length of extremal rays, it suffices to show that there exists a uniform $t>0$ such that $(X,(1+t)B)$ is klt. The existence of such $t$ is implied by LCTB$_n$. However, since we only need to consider those $C$ not contracted by $f$, by using length of extremal rays in a tricky way, it turns out that LCTB$_{n-1}$ is sufficient, see Lemma \ref{lemma ample}.

Finally, if $|-rK_X+mf^*(H)|$ gives a birational map on $X$ for fixed $r$ and $m$, we need to bound its volume. The idea is basically the same with \cite{Jiang}. According to the volume, we may construct non-klt centers, and restrict on the general fiber of $f$. By considering the lower bound of $\alpha$-invariants on the general fiber (GA$_{\leq n-1}$), we can get the upper bound of volumes, see Theorem \ref{thm volume}.

\bigskip
{\it Acknowledgments.} The author would like to thank Professor Chenyang Xu for suggesting this problem and discussion.  
The author is indebted to Professor Yujiro Kawamata  for effective conversations.


\section{Preliminaries}
\subsection{Notation and conventions}\label{section notation}
 We adopt the standard notation and definitions in \cite{KMM} and \cite{KM}, and will freely use them.

A {\it pair} $(X, B)$ consists of  a normal projective variety $X$ and an effective
$\mathbb{Q}$-divisor $B$ on $X$ such that
$K_X+B$ is $\mathbb{Q}$-Cartier.   

The variety $X$ is called a
\emph{Fano variety} if $-K_X$ is ample. The pair $(X, B)$ is called a
\emph{log Fano pair} if $-(K_X+B)$ is ample.

Let $f: Y\rightarrow X$ be a log
resolution of the pair $(X, B)$, write
$$
K_Y =f^*(K_X+B)+\sum a_iF_i,
$$
where $\{F_i\}$ are distinct prime divisors. The coefficient $a_i$ is called the {\it discrepancy} of $F_i$ with respect to $(X, B)$, and denoted by $a_{F_i}(X, B)$. For some $\epsilon \in [0,1]$, the
pair $(X,B)$ is called
\begin{itemize}
\item[(a)] \emph{$\epsilon$-kawamata log terminal} (\emph{$\epsilon$-klt},
for short) if $a_i> -1+\epsilon$ for all $i$;

\item[(b)] \emph{$\epsilon$-log canonical} (\emph{$\epsilon$-lc}, for
short) if $a_i\geq  -1+\epsilon$ for all $i$;

\item[(c)] \emph{terminal} if  $a_i> 0$ for all $f$-exceptional divisors $F_i$ and all $f$. 
\end{itemize}
Usually we write $X$ instead of $(X,0)$ in the case $B=0$.
Note that $0$-klt (resp. $0$-lc) is just klt (resp. lc) in the usual sense. 
   $F_i$ is called a {\it non-klt place}  (resp. {\it non-lc place}) of $(X, B)$  if $a_i\leq -1$ (resp. $<-1$).
A subvariety $V\subset X$ is called a {\it non-klt center} (resp. {\it non-lc center}) of $(X, B)$ if it is the image of a non-klt place (resp. non-lc place). The {\it non-klt locus} $\text{Nklt}(X, B)$ is the union of all  non-klt centers of $(X, B)$. The {\it non-lc locus} $\text{Nlc}(X, B)$ is the union of all  non-lc centers of $(X, B)$. 

In particular, 
a normal projective variety $X$ is \emph{of $\epsilon$-Fano type} if there exists an effective $\bQ$-divisor $B$ such that $(X, B)$ is an $\epsilon$-klt log Fano pair.

Let $(X, B)$ be  an lc pair and $D\geq 0$ be a $\bQ$-Cartier $\bQ$-divisor. The
{\it log canonical threshold} of $D$ with respect to $(X, B)$ is
$$\lct(X, B; D) = \sup\{t\in \bQ \mid (X, B+ tD) \text{ is lc}\}.$$
For application, we need to consider the case when $D$ is not effective. 
Let $G$ be a $\bQ$-Cartier $\bQ$-divisor satisfying $G+B\geq 0$, the
{\it unusual log canonical threshold} of $G$ with respect to $(X, B)$ is
$$\ulct(X, B; G) = \sup\{t\in [0,1] \cap \bQ \mid (X, B+ tG) \text{ is lc}\}.$$
Note that the assumption $t\in [0,1]$ guarantees that $B+ tG\geq 0$.

If  $(X, B)$ is an lc log Fano pair,  the {\it  (unusual)  $\alpha$-invariant} of $(X, B)$ is defined by 
$$
\alpha(X, B)=\inf\{\ulct(X,B;G)\mid G\sim_\bQ-(K_X+B), G+B\geq 0\}.
$$

A collection of varieties $\{X_t\}_{t\in T}$ is
said to be \emph{bounded} (resp.  \emph{birationally bounded}) if there exists $h:\mathcal{X}\rightarrow
S$ a projective morphism between schemes of finite type such that 
each $X_t$  is isomorphic (resp. birational) to $\mathcal{X}_s$ for some $s\in S$.

\subsection{$(n, d, a, \epsilon)$-Fano fibrations}
We define Fano fibrations, Mori fibrations, and $(n, d, a, \epsilon)$-Fano fibrations.
\begin{definition}A projective morphism $f:X\to Z$ between normal projective varieties is called a {\it Fano fibration} if
\begin{enumerate}
\item $X$ is  with terminal singularities; 

\item $f$ is a {\it contraction}, i.e., $f_*\OO_X=\OO_Z$;

\item $-K_X$ is ample over $Z$;

\item $\dim X > \dim Z$. 
\end{enumerate}
A Fano fibration $X\to Z$ is said to be a {\it Mori fibration} if $X$ is $\mathbb{Q}$-factorial and $\rho(X/Z)=1$.
\end{definition}
Note that a general fiber of a Fano fibration is a terminal Fano variety.
\begin{definition}\label{ndae}Fix positive integers $n$ and $d$, a rational number $a\geq 0$, and $0\leq \epsilon\leq 1$.
An {\it $(n, d, a, \epsilon)$-Fano fibration} $(f:X\to Z, B, H)$ consists of a Fano fibration  $f:X\to Z$, an effective $\bQ$-Cartier  $\bQ$-divisor $B$  on $X$, and a very ample divisor  $H$ on $Z$ such that 
\begin{enumerate}
\item $\dim X=n$;
\item $(H^{\dim Z})= d$;

\item $-(K_X+B)\sim_\bQ A-af^*(H)$ where $A$ is an ample $\bQ$-divisor on $X$; 

\item $(X, B)$ is $\epsilon$-klt.

\end{enumerate}
It is said to be {\it almost-extremal} if moreover,
\begin{enumerate}
\item[(5)] if we write $B=B'+B''$ where every component of $B'$ dominates $Z$ and  every component of $B''$ does not dominate $Z$, then $B''\sim_{\bQ} 0$ over $Z$.
\end{enumerate}
Here if $\dim Z=0$, we always set $H=0$ and $d=(H^{\dim Z})=1$. 
\end{definition}

\begin{remark}Condition (5) seems to be a little technical, but very natural.
Condition (5) holds if either $\dim Z=0$ or $X$ is $\bQ$-factorial and $\rho(X/Z)=1$. 
In this paper, condition (5) will be only used for LCTB$_n$. Note that if $f$ is an extremal contraction induced by an extremal ray, then $\rho(X/Z)=1$. This is the motivation of defining the terminology ``almost-extremal".
\end{remark}
In particular, if $(X,B)$ is an $\epsilon$-klt log Fano pair with a Mori fibration $f:X\to Z$, then $(f:X\to Z, B, H)$ is naturally an almost-extremal $(n, d, 0, \epsilon)$-Fano fibration for any very ample divisor $H$ on $Z$ with $d=(H^{\dim Z})$.

\begin{remark}\label{const}Suppose $(f:X\to Z, B, H)$ is an  $(n, d, a, \epsilon)$-Fano fibration. Then 
a general fiber $F$ of $f$ is a terminal Fano variety and $(F, B|_F)$ is an $\epsilon$-klt log Fano pair. Hence $(F\to f(F), B|_F, 0)$ is naturally an almost-extremal $(n-m, 1, 0, \epsilon)$-Fano fibration where $m=\dim Z$.

Suppose that $\dim Z>1$. Take a general element  $Z_1\in |H|$.  Denote $X_1=f^*(Z_1)\in |f^*(H)|$, $B_1=B|_{X_1}$, and $H_1= H|_{Z_1}$. Then $X_1$ and $Z_1$ are projective normal varieties, $X_1$ is again terminal, and the induced map $f_1:X_1\to Z_1$ is a Fano fibration. Moreover, $(H_1^{\dim Z_1})=(H^{\dim Z})=d$, 
$$
-(K_{X_1}+B_1)= -(K_X+B+X_1)|_{X_1}\sim_\bQ A|_{X_1}-(a+1)f_1^*(H_1),
$$
and $(X_1, B_1)$ is $\epsilon$-klt (cf. \cite[Lemma 5.17]{KM}).
Hence $(f_1: X_1\to Z_1, B_1, H_1)$ is an $(n-1, d, a+1, \epsilon)$-Fano fibration. Note that if $(f:X\to Z, B, H)$ is almost-extremal, so is $(f_1: X_1\to Z_1, B_1, H_1)$ since
$B_1=B'|_{X_1}+B''|_{X_1}$ satisfies condition (5) for general
 $X_1$ in $|f^*(H)|$.
\end{remark}

\subsection{Volumes}

Let $X$ be an $n$-dimensional projective variety  and $D$ be a Cartier divisor on $X$. The {\it volume} of $D$ is the real number
$$
{\Vol}(X, D)=\limsup_{m\rightarrow \infty}\frac{h^0(X,\OO_X(mD))}{m^n/n!}.
$$
Note that the limsup is actually a limit. Moreover by the homogenous property of  volumes, we can extend the definition to $\bQ$-Cartier $\bQ$-divisors. Note that if $D$ is a nef $\bQ$-divisor, then $\Vol(X, D)=D^n$. If $D$ is a non-$\bQ$-Cartier $\bQ$-divisors, we may take a $\bQ$-factorialization of $X$, i.e., a birational morphism $\phi:Y\to X$ which is isomorphic in codimension one and $Y$ is $\bQ$-factorial, then $\Vol(X, D):=\Vol(Y, \phi^{-1}_*D)$. Note that $\bQ$-factorializations always exist for klt pairs (cf. \cite[Theorem 1.4.3]{BCHM}).

For more background on volumes, see \cite[2.2.C, 11.4.A]{Positivity2}. It is easy to see the following inequality for volumes.
\begin{lem}\label{lemma volume}
Let $X$ be a projective normal variety, $D$ a $\bQ$-Cartier $\bQ$-divisor and $S$ a base-point free Cartier normal prime divisor. Then for any rational number $q>0$,
$$
\Vol(X,D+qS)\leq \Vol(X, D)+q(\dim X) \Vol(S, D|_S+qS|_S).
$$
\end{lem}
\begin{proof}For sufficiently divisible $m$ such that $mq\in \mathbb{Z}$ and an integer $k$,
consider the short exact sequence
$$
0\to \OO_X(mD+(k-1)S)\to \OO_X(mD+kS)\to \OO_S(mD|_S+kS|_S)\to 0.
$$
Then
\begin{align*}
{}& h^0(X, \OO_X(mD+kS))\\
\leq{}& h^0(X, \OO_X(mD+(k-1)S))+h^0(S, \OO_S(mD|_S+kS|_S)).
\end{align*}
Hence we have
\begin{align*}
{}&h^0(X, \OO_X(mD+mqS))\\\leq {}&h^0(X, \OO_X(mD))+\sum_{k=1}^{mq}h^0(S, \OO_S(mD|_S+kS|_S))\\
\leq {}&h^0(X, \OO_X(mD))+mq h^0(S, \OO_S(mD|_S+mqS|_S)).
\end{align*}
For the last step we use the assumption that $S$ is base-point free which implies that $S|_S$ is linearly equivalent to an effective divisor on $S$.
Dividing by $\frac{m^{\dim X}}{(\dim X)!}$ and taking limit, we get the inequality.
\end{proof}

\subsection{Non-klt centers, connectedness lemma, and inversion of adjunction}

The following lemma suggests a standard way to construct non-klt centers.
\begin{lem}[{cf. \cite[Lemma 2.29]{KM}}]\label{dimk}
Let $(X, B)$ be a pair and $V\subset X$ a closed subvariety of codimesion $k$ such that $V$ is not contained in the singular locus of $X$. If $\mult_V B\geq k$, then $V$ is a non-klt center of $(X, B)$. 
\end{lem}
Recall that the {\it multiplicity} $\mult_VF$ of a divisor $F$ along a subvariety $V$ is defined by the multiplicity $\mult_xF$ of $F$ at a general point $x\in V$.

Unfortunately, the converse of Lemma \ref{dimk} is not true unless $k=1$. Usually we do not have good estimates for the multiplicity along a non-klt center but the following lemma.

\begin{lem}[{cf. \cite[Theorem 9.5.13]{Positivity2}}]
Let $(X, B)$ be a pair and $V\subset X$  a non-klt center of $(X, B)$ such that $V$ is not contained in the singular locus of $X$. Then $\mult_V B\geq 1$. 
\end{lem}

If we assume some simple normal crossing condition on the boundary, we can get more information on the multiplicity along a non-klt center. For simplicity, we just consider surfaces.
\begin{lem}[{cf. \cite[4.1 Lemma]{McK}}]\label{multi}
Fix $0<e<1$. Let $S$ be a smooth surface, $B$  an effective $\bQ$-divisor, and $D$  a (not necessarily effective) simple normal crossing supported $\bQ$-divisor. Assume that  coefficients of $D$ are at most $e$ and $\mult_P B\leq 1-e$ for some point $P$, then for arbitrary divisor $E$ centered on $P$ over $S$, $a_E(S, B+D)\geq -e$. 
In particular, if $V$ is a non-klt center of $(S, B+D)$ and coefficients of $D$ are at most $e$, then $\mult_V B> 1-e$.
\end{lem}
\begin{proof}
By taking a sequence of point blow-ups, we can get the divisor $E$. Consider the blow-up at $P$, we have $f:S_1\rightarrow S$ with $K_{S_1}+B_1+D_1+mE_1=f^*(K_S+B+D)$ where $B_1$ and $D_1$ are the strict transforms of $B$ and $D$ respectively, and $E_1$ is the exceptional divisor with coefficient $m=\mult_P(B+D)-1\leq 1-e+2e-1=e$ since $\mult_P(D)\leq 2e$. 
Now $D_1+mE_1$ is again simple normal crossing supported and $\mult_QB_1\leq \mult_PB$ for $Q\in E_1$. Hence by induction on the number of blow-ups, we conclude that  the coefficient of $E$ is at most $e$ and hence $a_{E}(S, B+D)\geq -e$. 
\end{proof}

We have the following connectedness lemma of Koll\'{a}r and Shokurov for non-klt locus (cf.  
Shokurov \cite{Shokurov}, Koll\'{a}r \cite[17.4]{Kol92}).
\begin{thm}[Connectedness Lemma]
Let $f:X\rightarrow Z$ be a proper morphism of normal varieties with connected fibers and $D$ a $\bQ$-divisor such that $-(K_X+D)$ is $\bQ$-Cartier, $f$-nef, and $f$-big. Write $D=D^+-D^-$ where $D^+$ and $D^-$ are effective with no common components. If $D^-$ is $f$-exceptional (i.e. all of its components have image of codimension at least $2$), then ${\rm Nklt} (X,D)\cap f^{-1}(z)$ is connected for any $z\in Z$. 
\end{thm}
\begin{remark}
There are two main cases of interest of Connectedness Lemma:
\begin{enumerate}
\item $Z$ is a point and $(X,D)$ is a log Fano pair. Then $\Nklt(X,D)$ is connected.
\item $f:X\rightarrow Z$ is birational, $(Z,B)$ is a log pair and $K_X+D=f^*(K_Z+B)$.
\end{enumerate}
\end{remark}
As an application, we have the following theorem on inversion of adjunction (cf. \cite[Theorem 5.50]{KM}). Here we only use a weak version.
\begin{thm}[Inversion of adjunction]
Let $(X, B)$ be a pair and $S\subset X$ a normal  Cartier divisor not contained in the support of $B$. Then $$\Nklt(X,B)\cap S\subset \Nklt(S, B|_S).$$ In particular, if $\Nklt(X, B)\cap S\neq \emptyset$, then $(S, B|_S)$ is not klt.
\end{thm}

\subsection{Length of extremal rays}
Recall the result on length of extremal rays due to Kawamata.
\begin{thm}[{\cite{Ka}}]\label{ext ray}
Let $(X, B)$ be a klt pair. Then every $(K_X+B)$-negative extremal ray $R$ is generated by a rational curve $C$ such that $$0<-(K_X+B)\cdot C\leq 2\dim X.$$
\end{thm}
However, we need to deal with non-klt pairs in application. We need a generalization of this theorem for general pairs which is proved by Fujino.
\begin{thm}[{\cite[Theorem 1.1(5)]{Fujino}}]\label{ext ray2}
Let $(X, B)$ be a pair. Fix a $(K_X+B)$-negative extremal ray $R$. Assume that 
$$
R\cap \overline{NE}(X)_{\Nlc(X,B)} = \{0\},
$$
where 
$$
 \overline{NE}(X)_{\Nlc(X,B)} = {\rm Im}(\overline{NE}(\Nlc(X,B)) \to \overline{NE}(X)).
$$
Then $R$ is generated by a rational curve $C$ such that $$0<-(K_X+B)\cdot C\leq 2\dim X.$$
\end{thm}




\section{Boundedness of birationality}
In this section, we prove the first part of Theorem \ref{thm main}, on boundedness of ampleness and birationality.

\begin{thm}\label{thm birationality}
Fix integers $n>m>0$ and $d>0$, a rational number $a\geq 0$, $0<\epsilon<1$. Assume  {\rm LCTB}$_{n-1}$ and ${\rm BrTBAB}_{n-m}$ hold.
Then there exists an integer  $N(n,d, a, \epsilon)$ depending only on $n$, $d$, $a$, and  $\epsilon$, satisfying the following property: 

If $(f:X\to Z, B, H)$ is an almost-extremal $(n, d, a, \epsilon)$-Fano fibration with $\dim Z=m$, then $|-r_{n-m}K_X+N(n, d, a, \epsilon)f^*(H)|$ is ample and gives a birational map.
\end{thm}
The proof splits into two parts as the following two lemmas.

\begin{lem} \label{lemma ample}
Fix integers $n>0$ and $d>0$, a rational number $a\geq 0$, $0<\epsilon<1$. Assume  {\rm LCTB}$_{n-1}$ holds.
Then there exists a number $N'(n,d, a, \epsilon)$ depending only on $n$, $d$, $a$, and  $\epsilon$, satisfying the following property: 

If $(f:X\to Z, B, H)$ is an almost-extremal $(n, d, a, \epsilon)$-Fano fibration with $\dim Z>0$,  then $-K_X+kf^*(H)$ is ample for all $k\geq N'(n,d, a, \epsilon)$.
\end{lem}
\begin{proof}Let $(f:X\to Z, B, H)$ be an almost-extremal $(n, d, a, \epsilon)$-Fano fibration  with $\dim Z>0$. According to  {\rm LCTB}$_{n-1}$, we may take
$$
t_0=\min\{\lambda(n-1,d, a+1, \epsilon), \lambda(n-1,1, 0, \epsilon)\}>0.
$$

If $\dim Z=1$, for a general fiber $F$ of $f$, $(F\to f(F), B|_F, 0)$ is naturally an almost-extremal $(n-1, 1,0,\epsilon)$-Fano fibration by Remark \ref{const}. Hence by {\rm LCTB}$_{n-1}$, $(F, (1+t_0)B|_F)$ is klt. 

If $\dim Z>1$, by Remark \ref{const}, $(f_1:X_1\to Z_1, B_1, H_1)$ is an almost-extremal $(n-1, d,a+1,\epsilon)$-Fano fibration. Hence by {\rm LCTB}$_{n-1}$, $(X_1, (1+t_0)B_1)$ is klt. 

Hence, in either case, every curve in $\Nklt(X, (1+t_0)B)$ is contracted by $f$ by inversion of adjunction, which means that $f(\Nklt(X, (1+t_0)B))$ is a set of finitely many points. In particular, every curve $C_0$ supported in $\Nklt(X, (1+t_0)B)$ satisfies that $f^*(H)\cdot C_0 =0$. This implies that every class $c\in \overline{NE}(X)_{\Nlc(X,B)}$ satisfies that $f^*(H)\cdot c=0$ since $\Nlc(X, (1+t_0)B)\subset \Nklt(X, (1+t_0)B)$.

Now we consider an extremal ray $R$ of $\overline{NE}(X)$. 

If $R$ is $(K_X+(1+t_0)B)$-non-negative, then 
$$
\Big(-K_X+a\big(1+\frac{1}{t_0}\big)f^*(H)\Big)\cdot R= \big(1+\frac{1}{t_0}\big)A\cdot R+\frac{1}{t_0}(K_X+(1+t_0)B))\cdot R> 0.
$$

If $R$ is $(K_X+(1+t_0)B)$-negative and $f^*(H)\cdot R=0$, then $-K_X\cdot R>0$ since $-K_X$ is ample over $Z$.

If $R$ is $(K_X+(1+t_0)B)$-negative and $f^*(H)\cdot R>0$, then 
$$
R\cap \overline{NE}(X)_{\Nlc(X,B)} = \{0\}
$$
since we showed that $f^*(H)\cdot c =0$ for every class $c\in \overline{NE}(X)_{\Nlc(X,B)}$.
 By Theorem \ref{ext ray2}, $R$ is generated by a rational curve $C$ such that
$$
(K_X+(1+t_0)B)\cdot C \geq -2n.
$$
On the other hand, $f^*(H)\cdot C\geq 1$.
Hence 
\begin{align*}
{}&\Big(-K_X+\big(a+\frac{a+2n}{t_0}\big)f^*(H)\Big)\cdot C\\
={}& \big(1+\frac{1}{t_0}\big)A\cdot C+\frac{1}{t_0}(K_X+(1+t_0)B)\cdot C+\frac{2n}{t_0}f^*(H)\cdot C> 0.
\end{align*}

In summary, 
$$
(-K_X+kf^*(H))\cdot R> 0
$$
holds for every extremal ray $R$ and for all $k\geq a+\frac{a+2n}{t_0}$. By Kleiman's Ampleness Criterion, 
$-K_X+kf^*(H)$ is ample for all $k\geq a+\frac{a+2n}{t_0}$. We may take 
$$
N'(n,d,a,\epsilon)=a+\frac{a+2n}{\min\{\lambda(n-1,d, a+1, \epsilon), \lambda(n-1,1, 0, \epsilon)\}}
$$
and complete the proof.
\end{proof}

\begin{lem}\label{lemma birational}
Fix integers $n>m>0$, and $L>0$. Assume ${\rm BrTBAB}_{n-m}$ holds.
If $f:X\to Z$ is a Fano fibration with $\dim X=n$ and $\dim Z=m$, $H$ is a very ample divisor on $Z$ such that $-K_X+Lf^*(H)$ is ample, then $|-r_{n-m}K_X+kf^*(H)|$ gives a birational map for all $k\geq (r_{n-m}+1)L+2n-2$.
\end{lem}

\begin{proof} Let $f:X\to Z$ be a Fano fibration with $\dim X=n$ and $\dim Z=m$, $H$ a very ample divisor on $Z$ such that $-K_X+Lf^*(H)$ is ample.

If $m=\dim Z=1$, take a general fiber $F$ of $f$, then $F$ is a  terminal Fano variety of dimension $n-m$. By ${\rm BrTBAB}_{n-m}$, $|-r_{n-m}K_F|$ gives a birational map.
For two general fibers $F_1$ and $F_2$, for an integer $k\geq (r_{n-m}+1)L+2$, consider the short exact sequence
\begin{align*}
0\to {}&\OO_X(-r_{n-m}K_X+kf^*(H)-F_1-F_2)\to \OO_X(-r_{n-m}K_X+kf^*(H))\\
\to{}& \OO_{F_1}(-r_{n-m}K_{F_1})\oplus   \OO_{F_2}(-r_{n-m}K_{F_2})\to 0.
\end{align*}
Since $k\geq (r_{n-m}+1)L+2$, 
$$-(r_{n-m}+1)K_X+kf^*(H)-F_1-F_2$$
is ample, hence by Kawamata--Viehweg vanishing theorem,
\begin{align*}
{}&H^1(X, \OO_X(-r_{n-m}K_X+kf^*(H)-F_1-F_2))\\
={}&H^1(X, \OO_X(K_X-(r_{n-m}+1)K_X+kf^*(H)-F_1-F_2))=0.
\end{align*}
Hence 
\begin{align*}
{}& H^0(X,\OO_X(-r_{n-m}K_X+kf^*(H)))\\
\to {}&  H^0(F_1, \OO_{F_1}(-r_{n-m}K_{F_1}))\oplus H^0(F_2,  \OO_{F_2}(-r_{n-m}K_{F_2}))\end{align*}
is surjective. Since $|-r_{n-m}K_{F_i}|$ gives a birational map on $F_i$ for $i=1,2$, $|-r_{n-m}K_X+kf^*(H)|$ gives a birational map on $X$ for all $k\geq (r_{n-m}+1)L+2$.
Note that $n\geq 2$,  the lemma is proved in this case.

Now suppose that $\dim Z>1$. 
Recall the construction in Remark \ref{const}, take a general element  $Z_1\in |H|$,  denote $X_1=f^*(Z_1)\in |f^*(H)|$, $B_1=B|_{X_1}$, and $H_1= H|_{Z_1}$. Then $f_1:X_1\to Z_1$ is a Fano fibration with $\dim X_1=n-1$ and $\dim Z_1=m-1$, $H_1$ is a very ample divisor on $Z_1$ such that 
$$-K_{X_1}+(L+1)f_1^*(H_1)\sim_\bQ (-K_{X}+Lf^*(H))|_{X_1}$$
is ample.
By induction on $m$, we may assume that 
$|-r_{n-m}K_{X_1}+kf_1^*(H_1)|$ gives a birational map for all $k\geq (r_{n-m}+1)(L+1)+2n-4$.

For an  integer $k\geq (r_{n-m}+1)L+2n-3$, consider the short exact sequence
\begin{align*}
0\to {}&\OO_X(-r_{n-m}K_X+(k-1)f^*(H))\to \OO_X(-r_{n-m}K_X+kf^*(H))\\
\to{}&
 \OO_{X_1}(-r_{n-m}K_{X_1}+(k+r_{n-m})f_1^*({H_1}))\to 0.
\end{align*}
Since $k\geq (r_{n-m}+1)L+1$, $-(r_{n-m}+1)K_X+(k-1)f^*(H)$ is ample, hence by Kawamata--Viehweg vanishing theorem,
\begin{align*}
{}&H^1(X, \OO_X(-r_{n-m}K_X+(k-1)f^*(H)))\\
={}&H^1(X, \OO_X(K_X-(r_{n-m}+1)K_X+(k-1)f^*(H)))=0.
\end{align*}
Hence 
\begin{align*}
{}& H^0(X,\OO_X(-r_{n-m}K_X+kf^*(H)))\\
\to {}& H^0(X_1, \OO_{X_1}(-r_{n-m}K_{X_1}+(k+r_{n-m})f_1^*({H_1})))
\end{align*}
is surjective. By induction hypothesis,   $|-r_{n-m}K_{X_1}+(k+r_{n-m})f_1^*({H_1})|$ gives a birational map on $X_1$ since 
$$k+r_{n-m}\geq (r_{n-m}+1)(L+1)+2n-4.$$ 
In particular, $|-r_{n-m}K_X+kf^*(H)|\neq \emptyset$. 
Hence $|-r_{n-m}K_X+(k+1)f^*(H)|$ can separate general elements in $|f^*(H)|$, and 
$$
|-r_{n-m}K_X+(k+1)f^*(H)||_{X_1}=|-r_{n-m}K_{X_1}+(k+1+r_{n-m})f_1^*({H_1})|
$$
gives a birational map on $X_1$, which is a general element in $|f^*(H)|$. This implies that $|-r_{n-m}K_X+(k+1)f^*(H)|$ gives a birational map for all $k\geq (r_{n-m}+1)L+2n-3$.

We complete the proof.
\end{proof}

\begin{proof}[Proof of Theorem \ref{thm birationality}]
It follows from Lemmas \ref{lemma ample} and \ref{lemma birational}. We may take 
$
N(n,d,a,\epsilon)=(r_{n-m}+1)\roundup{N'(n,d,a,\epsilon)}+2n-2.
$
\end{proof}

\section{Boundedness of volumes}
In this section, we prove the second part of Theorem \ref{thm main}, on boundedness of volumes. We follow the idea in \cite{Jiang}.
\begin{thm}\label{thm volume}Fix integers $n>m> 0$ and $d>0$, a rational number $a\geq 0$, $0<\epsilon<1$. Assume {\rm GA}$_{n-m}$ holds.  Then there exists a number  $V'(n,d, a, \epsilon, k)$ depending only on $n$, $d$, $a$, $\epsilon$, and $k\in \mathbb{Z}_{\geq 0}$, satisfying the following property: 

If $(f:X\to Z, B, H)$ is an $(n, d, a, \epsilon)$-Fano fibration with $\dim Z=m$,
then $\Vol(X, -K_X+kf^*(H))\leq V'(n, d, a, \epsilon, k)$.

\end{thm}

\begin{remark}\label{rem ga}
Before proving the theorem, we remark that  {\rm GA}$_{n}$ implies the boundedness of anti-canonical volumes of terminal Fano variety of dimension $n$. In fact, let $Y$ be a terminal Fano variety of dimension $n$, then $(Y, 0)$ is a $\frac{1}{2}$-klt log Fano pair. By {\rm GA}$_{n}$, $\alpha(Y,0)\geq \mu(n,\frac{1}{2})>0$. On the other hand, it is well-known that $\alpha(Y,0)\cdot \sqrt[n]{(-K_Y)^n}\leq n$ (cf. \cite[6.7.1]{SOP}). Hence $(-K_Y)^n$ is bounded from above uniformly. We denote the bound to be $M'_0(n)$.

\end{remark}
\begin{proof}[Proof of Theorem \ref{thm volume}]
Let $(f:X\to Z, B, H)$ be an $(n, d, a, \epsilon)$-Fano fibration.

If $m=\dim Z=1$,  for a general fiber $F$  of $f$, $F$ is a  terminal Fano variety of dimension $n-m$. By Remark \ref{rem ga},  $\Vol(F, -K_F)\leq M'_0(n-m)$.
Fix $k\geq 0$, assume that for some $w>0$,
$$\Vol(X, -K_X+kf^*(H))>n(dk+w)M'_0(n-m).$$
It suffices to find an upper bound for $w$.
We may assume that $w>2$. Note that $f^*(H)\sim_\bQ dF$.
By Lemma \ref{lemma volume},
\begin{align*}
{}&\Vol(X, -K_X-wF)\\
\geq{}& \Vol(X, -K_X+kf^*(H))-n(dk+w)\Vol(F, -K_F)>0.
\end{align*}
Hence there exists an effective $\bQ$-divisor $B'\sim_\bQ-K_X-wF$.
For two general fibers $F_1$ and $F_2$, consider the pair $(X, (1-s)B+sB'+F_1+F_2)$ where $s=\frac{ad+2}{ad+w}<1$. Note that 
$$-(K_{X}+ (1-s)B+sB'+F_1+F_2)\sim_\bQ(1-s)A$$
is ample.
By Connectedness Lemma,  
$\Nklt(X, (1-s)B+sB'+F_1+F_2)$ is connected. On the other hand, it contains $F_1\cup F_2$, hence contains a non-klt center dominating $Z$. By inversion of adjunction, 
$(F, (1-s)B|_F+sB'|_F)$ is not klt for a general fiber $F$. On the other hand, $(F, B|_F)$ is an $\epsilon$-klt log Fano pair of dimension $n-m$, $F$ is a  terminal Fano variety, and $B'|_F-B|_F\sim_\bQ-(K_F+B|_F)$.
Hence by {GA}$_{n-m}$, $$s\geq \ulct(F, B|_F; B'|_F-B|_F) \geq \mu(n-m, \epsilon).$$ Hence $w\leq \frac{ad+2}{\mu(n-m,\epsilon)}-ad$. In this case, we may take 
$$
V'(n, d, a, \epsilon, k)=n\Big(dk+ \frac{ad+2}{\mu(n-m,\epsilon)}-ad\Big)M'_0(n-m).
$$

Now assume that $\dim Z>1$. 
As constructed in Remark \ref{const}, $(f_1:X_1\to Z_1, B_1, H_1)$ is an $(n-1, d, a+1, \epsilon)$-Fano fibration.
By induction, we may assume that $\Vol(X_1, -K_{X_1}+kf_1^*(H_1))\leq V'(n-1, d, a+1, \epsilon, k)$.
Fix $k\geq 0$, assume that for some $w>0$,
$$\Vol(X, -K_X+kf^*(H))>n(k+w)V'(n-1, d, a+1, \epsilon, k+1).$$
It suffices to find an upper bound for $w$.
We may assume that $w>m+1$.
By Lemma \ref{lemma volume},
\begin{align*}
{}&\Vol(X, -K_X-wX_1)\\
\geq {}& \Vol(X, -K_X+kf^*(H))-n(k+w)\Vol(X_1, -K_X|_{X_1}+kf^*(H)|_{X_1})\\
={}& \Vol(X, -K_X+kf^*(H))-n(k+w)\Vol(X_1, -K_{X_1}+(k+1)f_1^*(H_1)) >0.
\end{align*}
Hence there exists an effective $\bQ$-divisor $B'\sim_\bQ-K_X- wX_1$.
Take $s=\frac{a+m+1}{a+w}<1$.
For a general fiber $F_1$ over $z_1\in Z$, then there exists a number $\delta>0$ (cf. \cite[4.8]{SOP}) such that for any general $H'\in |H|$ containing $z_1$, 
$$\Nklt(X, (1-s)B+sB')=\Nklt(X, (1-s)B+sB'+\delta f^*(H')).$$
We may take general $H^{j}\in |H|$ containing $z_1$ for $1\leq j\leq J$ with $J>\frac{m}{\delta}$  and take $G_1=\sum_{j=1}^J\frac{m}{J}f^*(H^{j})$. Then $\mult_{F_1}G_1\geq m$ and $G_1\sim_\bQ mf^*(H)\sim_\bQ mX_1$. In particular, $(X, G_1)$ is not klt at $F_1$ and by construction, in a neighborhood of $F_1$,
\begin{align*}
{}&\Nklt(X, (1-s)B+sB')\cup F_1\\
={}&\Nklt(X, (1-s)B+sB'+G_1).
\end{align*}
Take a general element $G_2\in |f^*(H)|$, consider the pair $(X, (1-s)B+sB'+G_1+G_2)$ where $s=\frac{a+m+1}{a+w}<1$. 
Then
$$-(K_{X}+ (1-s)B+sB'+G_1+G_2)\sim_\bQ(1-s)A$$
is ample.
Since 
$$F_1\cup G_2\subset \Nklt(X, (1-s)B+sB'+G_1+G_2),$$
 by Connectedness Lemma,   there is a curve $C$ contained in $\Nklt(X, (1-s)B+sB'+G_1+G_2)$, intersecting $F_1$ and not contracted by $f$. Hence $C$ is  contained in $\Nklt(X, (1-s)B+sB')$ by the construction of $G_1$ and generality of $G_2$. Since $C$ intersects $F_1$, so does $\Nklt(X, (1-s)B+sB')$. Since $F_1$ is  a general fiber over $Z$, $\Nklt(X, (1-s)B+sB')$ dominates $Z$. 
 By inversion of adjunction, 
$(F, (1-s)B|_F+sB'|_F)$ is not klt for a general fiber $F$. On the other hand, $(F, B|_F)$ is an $\epsilon$-klt log Fano pair of dimension $n-m$ and $F$ is a  terminal Fano variety.
Hence by {GA}$_{n-m}$,
 $$s\geq \ulct(F, B|_F, B'|_F-B|_F) \geq \mu(n-m, \epsilon).$$
Hence $w\leq \frac{a+m+1}{\mu(n-m,\epsilon)}-a$. We may take
$$
V'(n,d,a,\epsilon, k)=n\Big(k+\frac{a+m+1}{\mu(n-m,\epsilon)}-a\Big)V'(n-1, d, a+1, \epsilon, k+1)
$$
inductively, and complete the proof.
\end{proof}

\section{Lower bound of log canonical thresholds in dimension two}\label{section lct}
In this section, we consider LCTB$_2$. Firstly, we prove the following general theorem for surfaces. The basic idea of proof comes from \cite{AM} and \cite{Jiang}, but we are in a totally different situation from \cite{Jiang}.
\begin{thm}\label{thm surface}Fix $m>0$ and  $0<\epsilon<1$. Then there exists a number $\lambda'(m,\epsilon)>0$ depending only  on $m$ and $\epsilon$ satisfying the following property:

If $T$ is a projective smooth surface and $B=\sum_ib_iB^i$ an effective $\bQ$-divisor on $T$ such that 
\begin{enumerate}
\item $(T, B)$ is $\epsilon$-klt, but $(T, (1+t)B)$ is not klt for some $t>0$;
\item $K_T+B\sim_\bQ G-A$ where $A$ is an ample $\bQ$-divisor and $G$ is a nef $\bQ$-divisor on $T$;
\item  $\sum_ib_i\leq m$;
\item $(B)^2\leq m$, $B\cdot G\leq m$.
\end{enumerate}
Then $t>\lambda'(m,\epsilon)$.
\end{thm}
\begin{proof}
Take $(T, B)$ as in the theorem. 
Since $(T, B)$ is $\epsilon$-klt, $b_i<1-\epsilon$ for all $i$. We may assume that $t<\epsilon$ since we want to bound $t$ from below, and hence $(1+t)b_i<1$ for all $i$.
Since $(T, (1+t)B)$ is not klt, it has isolated non-klt centers. We may take a sequence of point blow-ups 
$$
T_{r}\rightarrow T_{r-1}\rightarrow \cdots \rightarrow T_{1}\rightarrow T_{0}=T
$$
where $T_{k+1}\rightarrow T_k$ is the blow-up at a non-klt center $P_k\in \Nklt(T_k, (1+t)B_k+E_k)$ where $B_k$
is the strict transform of $B$ on $T_k$ and 
$$
K_{T_k}+(1+t)B_k+E_k=\pi_k^*(K_T+(1+t)B),
$$
where $\pi_k:T_k\rightarrow T$ is the composition map and $E_k$ is a $\pi_k$-exceptional $\bQ$-divisor. For $k\geq l$, denote $\pi_{k,l}$ to be the composition map $T_k\to T_l$ and $E_k^l$ be the strict transform of the exceptional divisor $E^l$ of $\pi_{l,l-1}$ on $T_k$. Then we can write 
$E_k=\sum_{l=1}^ke_lE_k^l$.
For $l\geq 1$, since $P_{l-1}$ is a non-klt center of $(T_{l-1}, (1+t)B_{l-1}+E_{l-1})$, $$e_l=\mult_{P_{l-1}}((1+t)B_{l-1}+E_{l-1})- 1\geq 0.$$
We stop this process at $T_{r}$ if $\rounddown{E_r}\neq 0$.
Furthermore, we may assume that $\mult_{P_k}B_k$ is non-increasing. 
Write 
$$
K_{T_k}+B_k+E'_k=\pi_k^*(K_T+B),
$$
where $E'_k=\sum_{l=1}^ke'_lE_k^l$, note that $e'_l$ may be negative. 
Take the integer $s$ such that 
$$s=\max\{k\leq r\mid \mult_{P_{k-1}}B_{k-1}\geq  \frac{\epsilon}{2} \text{ and } e'_l>-\frac{\epsilon^2}{4} \text{ for all } l\leq k\}.$$
Recall that $B_s=\sum_i b_i B_s^i$ with $0\leq b_i<1-\epsilon$ where $B_s^i$ is the strict transform of $B^i$ on $T_s$ for all $i$.

\begin{claim}
$(B_{s}^i)^2\geq -\frac{2}{\epsilon}-\frac{1}{\epsilon}G\cdot B^i-\frac{\epsilon}{4}\sum_{l=1}^s\mult_{P_{l-1}}B_{l-1}^i$ for all $i$.
\end{claim}
\begin{proof}
Suppose that $(B_{s}^i)^2<0$, then
\begin{align*}
-2\leq{}&
2p_a(B_{s}^i)-2=(K_{T_{s}}+B_{s}^i)\cdot B_{s}^i\\
={}&{\epsilon} (B_{s}^i)^2+(K_{T_{s}}+(1-{\epsilon} )B_{s}^i)\cdot B_{s}^i\\
\leq{}&{\epsilon}  (B_{s}^i)^2+(K_{T_{s}}+ B_{s}+E'_{s})\cdot B_{s}^i-E'_{s}\cdot B_{s}^i\\
={}&{\epsilon} (B_{s}^i)^2+(K_T+B)\cdot  B^i-\sum_{l=1}^se'_lE_s^l\cdot B_{s}^i\\
\leq{}&{\epsilon} (B_{s}^i)^2+G\cdot B^i+\frac{\epsilon^2}{4}\sum_{l=1}^sE_s^l\cdot B_{s}^i\\
\leq{}&{\epsilon} (B_{s}^i)^2+G\cdot B^i+\frac{\epsilon^2}{4}\sum_{l=1}^sE_s^l\cdot \pi_{s, l}^*B_{l}^i\\
={}&{\epsilon} (B_{s}^i)^2+G\cdot B^i+\frac{\epsilon^2}{4}\sum_{l=1}^sE^l\cdot B_{l}^i\\
={}&{\epsilon}(B_{s}^i)^2+G\cdot B^i+\frac{\epsilon^2}{4}\sum_{l=1}^s\mult_{P_{l-1}}B_{l-1}^i.
\end{align*}
Hence we proved the claim.
\end{proof}

Now we can give an upper bound for $s$.

On $T_{s}$, we have
\begin{align*}
(B_{s})^2={}&(\sum_i b_iB_{s}^i)^2\geq \sum_i b_i^2(B_{s}^i)^2\\
\geq  {}&  \sum_i b_i^2\Big(-\frac{2}{\epsilon}-\frac{1}{\epsilon}G\cdot B^i-\frac{\epsilon}{4}\sum_{l=1}^s\mult_{P_{l-1}}B_{l-1}^i\Big)\\
\geq  {}&  \sum_i b_i\Big(-\frac{2}{\epsilon}-\frac{1}{\epsilon}G\cdot B^i-\frac{\epsilon}{4}\sum_{l=1}^s\mult_{P_{l-1}}B_{l-1}^i\Big)\\
\geq {}&-\frac{2m}{\epsilon}- \frac{1}{\epsilon}G\cdot B -  \frac{\epsilon}{4}\sum_{l=1}^s\mult_{P_{l-1}}B_{l-1}\\
\geq {}&-\frac{3m}{\epsilon}-  \frac{\epsilon}{4}\sum_{l=1}^s\mult_{P_{l-1}}B_{l-1}.\end{align*}
On the other hand, 
$(B_{s})^2=(B)^2-\sum_{l=1}^s(\mult_{P_{l-1}}B_{l-1})^2$ and $(B)^2\leq m$. Hence
$$
m+ \frac{3m}{\epsilon} \geq \sum_{l=1}^s\mult_{P_{l-1}}B_{l-1}\big(\mult_{P_{l-1}}B_{l-1}-\frac{\epsilon}{4}\big)\geq \frac{\epsilon^2}{8}s
$$
by the assumption $\mult_{P_k}B_k\geq \frac{\epsilon}{2}$ for $k< s$. Hence 
$$
{s}< \frac{32m}{\epsilon^3}.
$$

\begin{claim}\label{QQQ}There exists a point $Q_s$ on $T_{s}$ such that  $\mult_{Q_s}\pi_{s}^*(tB)\geq \frac{\epsilon^2}{4}$.
\end{claim}
\begin{proof}
Consider the pair $(T_{s}, (1+t)B_{s}+E_{s})$.  Note that $E_{s}$ is simple normal crossing supported.

Assume that there exists a curve $E$ with coefficient at least $1-\frac{3\epsilon}{4}$ in $E_{s}$, that is, $$\mult_E(K_{T_{s}}-\pi_{s}^*(K_T+(1+t)B))\leq -1+\frac{3\epsilon}{4}.$$
 On the other hand, since $(T, B)$ is $\epsilon$-klt, 
$$\mult_E(K_{T_{s}}-\pi_{s}^*(K_T+B))> -1+\epsilon.$$
Hence $\mult_E\pi_{s}^*(tB)\geq  \frac{\epsilon}{4}\geq \frac{\epsilon^2}{4}$ and we can take any point $Q_s\in E$.

If all coefficients of $E_{s}$ are smaller than $1-\frac{3\epsilon}{4}$,  then $s<r$ and $P_{s}$ is a non-klt center of  $(T_{s}, (1+t)B_{s} +E_{s})$. By Lemma \ref{multi}, $\mult_{P_{s}}((1+t)B_{s}) \geq \frac{3\epsilon}{4}$.  Since we need a lower bound for $t$, we may assume that $t<\frac{1}{2}$, then $\mult_{P_{s}}B_{s} \geq  \frac{\epsilon}{2}$. Therefore, by the maximality of $s$,  $e'_{s+1}\leq -\frac{\epsilon^2}{4}$. Then 
$$
\mult_{P_{s}}\pi^*_s(tB)=\mult_{E_{s+1}}\pi^*_{s+1}(tB)=e_{s+1}-e'_{s+1}\geq \frac{\epsilon^2}{4}.
$$
We can take $Q_s=P_{s}$.

We proved the claim.
\end{proof}


By Claim \ref{QQQ}, it follows that $\mult_{Q_0}(tB)\geq \frac{\epsilon^2}{2^{s+2}}$ where $Q_0=\pi_s(Q_s)$ (cf. \cite[Section 5, Claim 4]{Jiang}).
On the other hand, since $(X, B)$ is klt, $\mult_{Q_0}(B)< 2$ by Lemma \ref{dimk}.
Combining with the inequality $s< \frac{32m}{\epsilon^3}$, we have
$$
t >  \frac{\epsilon^2}{2^{32m/\epsilon^3+3}},
$$
and hence we may take this number to be $\lambda'(m,\epsilon)$. 
\end{proof}

As an application, we can prove LCTB$_2$ now.
\begin{proof}[Proof of Theorem \ref{thm LCTB2}]
Let $(f:T\to Z, B, H)$ be an almost-extremal $(2, d, a, \epsilon)$-Fano fibration. There are two cases, $\dim Z=0$ or $1$.

(1) Suppose that $\dim Z=0$. Then $T$ is a smooth del Pezzo surface, 
 $(T, B)$ is $\epsilon$-klt, and $-(K_T+B)$ is ample.
Note that $-K_T$ is ample, $-3K_T$ is very ample, and $(-K_T)^2\leq 9$ for the del Pezzo surface $T$.
 Write $B=\sum_ib_iB^i$, then 
\begin{align*} 
\sum_i b_i\leq{}& B\cdot (-K_T)\leq (-K_T)^2\leq 9;\\
(B)^2\leq{}& (-K_T)^2\leq 9.
\end{align*}
Apply Theorem \ref{thm surface} for $G=0$, then $(T, (1+t)B)$ is klt for all $0<t\leq \lambda'(9, \epsilon)$.

(2) Suppose that $\dim Z=1$. By assumption, $(T, B)$ is $\epsilon$-klt and  $$K_T+B\sim_\bQ -A+af^*(H) \sim_\bQ -A+adF,$$ where $A$ is an ample $\bQ$-divisor on $T$ and $F$ is a general fiber of $f$. Moreover, $-K_T$ is ample over $Z$, that is, $f:T\to Z$ is a {\it conic bundle} such that all fibers are plane conics, a smooth fiber is a smooth rational curve, and a singular fiber is the union of two lines intersecting at one point. Note that the assumption that $(f:T\to Z, B, H)$ is almost-extremal implies that we may write $B=\sum_ib_iB^i+\sum_jc_jF^j$, where $B^i$ is a curve not contained in a fiber for all $i$, and $F^j$ is a whole fiber for all $j$. This condition is crucial in the following claim.
Recall that $B\cdot F<(-K_T)\cdot F=2$ and $(-K_T)^2\leq 8$ for the conic bundle $T$.

\begin{claim}\label{claim bi}
$\sum_ib_i+\sum_j 2 c_j\leq 2+\frac{8+4ad}{\epsilon}$. 
\end{claim}
\begin{proof}Firstly, we have
$$\sum_ib_i\leq \sum_ib_iB^i\cdot F=B\cdot F<2.$$
Hence 
it suffices to show that $\sum_jc_j\leq  \frac{4+2ad}{\epsilon}$. Assume, to the contrary,  that $w=\sum_jc_j>\frac{4+2ad}{\epsilon}$, then $B-wF\sim_\bQ D$ for some effective $\bQ$-divisor $D$ and for a general fiber $F$. For two general fibers $F_1$ and $F_2$, consider
$(T, (1-\frac{2+ad}{w})B+\frac{2+ad}{w}D+F_1+F_2)$, then 
$$-\Big(K_T+\big(1-\frac{2+ad}{w}\big)B+\frac{2+ad}{w}D+F_1+F_2\Big)\sim_\bQ A$$
is ample. Note that 
$$
F_1\cup F_2\subset \Nklt\Big(T, \big(1-\frac{2+ad}{w}\big)B+\frac{2+ad}{w}D+F_1+F_2\Big).
$$
By Connectedness Lemma, there is a curve $C$ in $\Nklt(T, (1-\frac{2+ad}{w})B+\frac{2+ad}{w}D+F_1+F_2)$ dominating $Z$. Hence 
$$
\mult_C \Big(\big(1-\frac{2+ad}{w}\big)B+\frac{2+ad}{w}D\Big)\geq 1.
$$
Since $(T, B)$ is $\epsilon$-klt, $\mult_{C}B< 1-\epsilon$, and hence $\mult_{C}\frac{2+ad}{w}D>\epsilon$. On the other hand, $\mult_{C}D\leq D\cdot F=B\cdot F<2$. Hence $w<\frac{4+2ad}{\epsilon}$, a contradiction.
\end{proof}

Also we have
\begin{align*}
(B)^2< {}&(B+A)^2=(adF-K_T)^2=4ad+(-K_T)^2\leq 4ad+8;\\
B\cdot adF< {}& (-K_T)\cdot adF =2ad.
\end{align*}


Apply Theorem \ref{thm surface} for $G=adF$ and $m=\frac{10+4ad}{\epsilon}$, then $(T, (1+t)B)$ is klt for all $0<t\leq \lambda'(\frac{10+4ad}{\epsilon}, \epsilon)$.
\end{proof}

\section{Proof of theorems and corollaries}
\begin{proof}[Proof of Theorem \ref{thm main}]It follows from Theorems \ref{thm birationality} and \ref{thm volume}. We may take $N(n,d,a,\epsilon)$ as in Theorem \ref{thm birationality}, and take 
$$
V(n,d,a,\epsilon)=r^n_{n-m}\cdot V'(n,d,a,\epsilon, N(n,d,a,\epsilon))
$$ 
as in Theorem \ref{thm volume}. The birational boundedness follows easily, cf. \cite[Lemma 2.4.2(2)]{HMX13}.
\end{proof}
Before proving the corollaries, we recall the following theorem proved in \cite{Jiang} by using Minimal Model Program.
\begin{thm}[{cf. \cite[Proof of Theorem 2.3]{Jiang}}]\label{thm MFS}Fix an integer $n>0$ and $0<\epsilon<1$. 
Every $n$-dimensional variety $X$ of $\epsilon$-Fano type is birational to an $n$-dimensional variety  $X'$ of $\epsilon$-Fano type  with a Mori fibration such that $\Vol(X, -K_X)\leq \Vol(X',-K_{X'}).
$
\end{thm}

\begin{proof}[Proof of Corollary \ref{cor BBAB}]
By Theorem \ref{thm MFS}, to prove BBAB$_n$, we only need to show the birational boundedness of varieties of $\epsilon$-Fano type  with a Mori fibration. 

Let $X$ be an $n$-dimensional variety of $\epsilon$-Fano type  with a Mori fibration $f: X\to Z$. By S$_n$, $Z$ is of $\delta(n, \epsilon)$-Fano type of dimension less than $n$. 

If $\dim Z=0$, then $X$ is an $n$-dimensional $\bQ$-factorial terminal Fano variety of Picard number one. There is nothing to prove since we assume BTBAB$_n$.

Suppose that $\dim Z>0$. Then $Z$ is bounded by BAB$_{\leq n-1}$. In particular, there exists a very ample divisor $H$ on $Z$ with degree $d$ such that $d$ is bounded from above  by some number $D(n, \epsilon)$ depending only on $n$ and $\epsilon$. By definition, there exists a $\bQ$-divisor $B$ such that $(X, B)$ is an $\epsilon$-klt log Fano pair. Hence $(f:X\to Z,B,H)$ is an almost-extremal $(n, d, 0,\epsilon)$-Fano fibration. Since BrTBAB$_{\leq n-1}$ is implied by BAB$_{\leq n-1}$, by Theorem \ref{thm main}, for fixed $d$, such $X$ forms a birationally bounded family. Since $d$ has only finitely many possible values, we complete the proof.
\end{proof}

\begin{proof}[Proof of Corollary \ref{cor WBAB}]
By Theorem \ref{thm MFS}, to prove WBAB$_n$, we only need to show the boundedness of anti-canonical volumes of varieties of $\epsilon$-Fano type  with a Mori fibration. 

Let $X$ be an $n$-dimensional variety of $\epsilon$-Fano type  with a Mori fibration $f: X\to Z$. By S$_n$, $Z$ is of $\delta(n, \epsilon)$-Fano type of dimension less than $n$. 

If $\dim Z=0$, then $X$ is an $n$-dimensional $\bQ$-factorial terminal Fano variety of Picard number one. There is nothing to prove since we assume WTBAB$_n$.

Suppose that $\dim Z>0$. Then $Z$ is bounded by BAB$_{\leq n-1}$. In particular, there exists a very ample divisor $H$ on $Z$ with degree $d$ such that $d$ is bounded from above  by some number $D(n, \epsilon)$ depending only on $n$ and $\epsilon$. By definition, there exists a $\bQ$-divisor $B$ such that $(X, B)$ be an $\epsilon$-klt log Fano pair. Hence $(f:X\to Z,B,H)$ is an almost-extremal $(n, d, 0,\epsilon)$-Fano fibration. By Theorem \ref{thm volume}, $\Vol(X, -K_X) \leq V'(n, d,0, \epsilon, 0)$. Since $d$ has only finitely many possible values, we complete the proof.
\end{proof}

\begin{proof}[Proof of Corollaries \ref{cor BBAB3} and \ref{cor WBAB3}]
BAB$_{2}$ was proved by Alexeev \cite{AK2}, BTBAB$_3$ and WTBAB$_3$ are implied by TBAB$_3$ which was proved by Kawamata \cite{K}, LCTB$_2$ holds by Theorem \ref{thm LCTB2}, and GA$_2$ was proved in \cite[Theorem 2.8]{Jiang}. 

Finally, we show that S$_3$ is implied by \cite[Corollary 1.7]{Birkar} (cf. \cite[Theorem 6.3]{Jiang}). Consider an $\epsilon$-klt log Fano pair $(X, B)$ of dimension $3$ with a Mori fibration $X\to Z$.   If $\dim Z\leq 1$, there is nothing to prove. Suppose that $\dim Z=2$, then there exist effective $\bQ$-divisors $\Delta$ and $\Delta'$ such that $(Z, \Delta)$ is klt, $-(K_Z+\Delta)$ is ample by \cite[Corollary 3.3]{FG}, and $(Z, \Delta')$ is $\delta$-klt, $-(K_Z+\Delta')\sim_\bQ 0$ by \cite[Corollary 1.7]{Birkar}. Note that $\delta$ depends only on $\epsilon$. We may choose sufficiently small $t>0$ such that $(Z, (1-t)\Delta'+t\Delta)$ is still $\delta$-klt. In this case, 
$$
-(K_Z+(1-t)\Delta'+t\Delta)\sim -t(K_Z+\Delta)
$$
is ample. Hence $Z$ is of $\delta$-Fano type.

As all the conjectures we need are confirmed in lower dimension, BBAB$_3$ and WBAB$_3$ hold by Corollaries \ref{cor BBAB} and \ref{cor WBAB}.
\end{proof}

\end{document}